\documentclass[12pt]{article}
\usepackage{amsfonts}
\usepackage{amsmath,color}
\usepackage[makeroom]{cancel}
\usepackage{hyperref}
\newcommand{\beq}{\begin{equation}}
\newcommand{\eeq}{\end{equation}}
\newcommand{\bea}{\begin{eqnarray}}
\newcommand{\eea}{\end{eqnarray}}
\newcommand{\beas}{\begin{eqnarray*}}
\newcommand{\eeas}{\end{eqnarray*}}

\newtheorem{theorem}{Theorem}[section]

\newtheorem{assumption}[theorem]{Assumption}
\newtheorem{definition}[theorem]{Definition}
\newtheorem{proposition}[theorem]{Proposition}

\newtheorem{corollary}[theorem]{Corollary}

\newtheorem{remark}[theorem]{Remark}
\newtheorem{example}[theorem]{Example}
\newtheorem{examples}[theorem]{Examples}
\newtheorem{foo}[theorem]{Remarks}

\newenvironment{proof}{\addvspace{\medskipamount}\par\noindent{\it
Proof}.}
{\unskip\nobreak\hfill$\Box$\par\addvspace{\medskipamount}}

\newcommand{\bM}{\mathbb M}

\newcommand{\V}{\mathcal V}

\newcommand{\M}{\mathbb M}

\newcommand{\R}{\mathbb R}

\newcommand{\Dv}{{\Delta}_{\mathcal{V}}}

\title{Wasserstein contraction properties for hypoelliptic diffusions}

\author{Fabrice Baudoin\footnote{Author supported in part by Grant NSF-DMS 15-11-328}}

\date{Department of Mathematics, Purdue University
}

\begin{document}
\maketitle

\begin{abstract}
Gradient bounds had proved to be a very efficient tool for the control of the rate of convergence to equilibrium for parabolic evolution equations.  Among the gradient bounds methods, the celebrated Bakry-\'Emery criterion is a powerful way prove to convergence to equilibrium with an exponential rate. To be satisfied, this criterion requires some form of ellipticity of the diffusion operator. In the past few years, there have been several works extending the Bakry-\'Emery methodology to hypoelliptic operators. Inspired by these methods, we describe a rather simple generalization of the criterion that applies to a large class of hypoelliptic/hypocoercive operators. We are particularly interested in convergence to equilibrium in the Wasserstein distance and obtain several new results in that direction.

\end{abstract}

\tableofcontents

\section{Introduction}

The study of convergence to equilibrium for diffusion semigroups is a topic that attracted of attention since the very beginning of the theory of Markov processes. Literature on this topic is wide and a large set of methods are available to prove convergence. Many of those methods combine tools from functional analysis, analysis of partial differential equations, probability theory, ergodic theory and differential geometry.

\

The method of gradient bounds developed by Bakry and \'Emery in their celebrated paper \cite{BE} is particularly powerful. If $L$ is a diffusion operator, we can associate to $L$ the \textit{carr\'e du champ} operator
 \[
\Gamma(f,g)=\frac{1}{2}\left(L(fg)- fLg-gLf\right)
\]
 and its iteration
 \[
 \Gamma_2(f,g)=\frac{1}{2}\left(L\Gamma(f,g)-\Gamma(f,Lg)-\Gamma(g,Lf)\right).
 \]

If the operator $L$ admits a symmetric measure $\mu$ and if for every $f$, $\Gamma_2(f,f)\ge \rho \Gamma(f,f)$ for some positive constant $\rho$, then it is known  the semigroup $P_t$ generated by $L$ will converge exponentially fast to an equilibrium (see \cite{BE}). The assumption of symmetry for $L$ is actually not strictly necessary. Indeed, if $\Gamma_2(f,f)\ge \rho \Gamma(f,f)$, then the following gradient bound holds
\begin{align}\label{lkjpo}
\Gamma(P_t f) \le e^{-2\rho t} P_t \Gamma(f)
\end{align}
where $P_t$ denotes the semigroup generated by $L$. Assuming further that
\[
d(x,y)=\sup \{ f(x)-f(y), \Gamma(f) \le 1 \}
\]
is a genuine distance, one may deduce from Kuwada's duality \cite{Kuwada} that the following Wasserstein contraction holds
\[
W_2 (P_t^* \mu , P_t^*\nu) \le e^{-\rho t} W_2 ( \mu , \nu)
\]
where $W_2$ is the $L^2$-Wasserstein distance associated with $d$. This implies both uniqueness of an invariant measure and convergence to equilibrium of $P_t$. However, the assumption $\Gamma_2(f,f)\ge \rho \Gamma(f,f)$ requires some form of ellipticity for the operator $L$ and is therefore  not satisfied for strictly hypoelliptic operators.

\

The idea for hypoelliptic operators is to replace the gradient bound \eqref{lkjpo} by a gradient bound of the form
\begin{align}\label{kjhgf}
\mathcal{T}(P_t f) \le e^{-2\rho t} P_t \mathcal{T} (f)
\end{align}
where $\mathcal{T}$ is a gradient which is a priori not intrinsically associated to $L$ as $\Gamma$ is. For the gradient bound \eqref{kjhgf}, Kuwada's duality still leads to uniqueness of an invariant measure and convergence to equilibrium and it tuns out that \eqref{kjhgf} is satisfied for a large class of hypoelliptic operators

\

For hypoelliptic operators, the idea of using gradient bounds similar to \eqref{kjhgf} is not new and has already been fruitfully used in different contexts (see for instance \cite{baudoin-bakry,BB,BG}). It seems however useful to state in one place general theorems in great generality and then detail some relevant examples.  Also, the use of Kuwada's duality to prove convergence to equilibrium is new in our general framework and leads to convergence results worth pointing out.

\

The paper is organized as follows. In Section 2, we formulate the generalization of the Bakry-\'Emery criterion that we illustrate then with several examples. In Section 3, we discuss an alternative approach that has been proposed in \cite{baudoin-bakry}.

\

\textbf{Acknowledgment:} \textit{The author would like to thank Luigi Ambrosio for stimulating discussions about Kuwada's duality.}

\section{Convergence in the Wasserstein distance}

\subsection{Generalization of the Bakry-\'Emery criterion} 

We consider on $\mathbb{R}^n$ a hypoelliptic diffusion operator 
\[
L=\sum_{i,j=1}^n a_{ij} (x) \frac{\partial^2}{\partial x_i \partial x_j } +\sum_{i=1}^n b_i (x) \frac{\partial}{\partial x_i},
\]
 with smooth coefficients $a_{ij},b_i$. Our basic problem is to find conditions ensuring convergence to an equilibrium for the heat semigroup generated by $L$. Such problem has already been widely addressed in the literature and several methods are available. We propose here a variant of the celebrated Bakry-\'Emery criterion which is particularly well-suited when dealing with hypoelliptic or hypocoercive diffusion operators.
 
 Let $\sigma_{ij}:\mathbb{R}^n \to \mathbb{R}$ be smooth functions such that for every $x \in \mathbb{R}^n$, the matrix $(\sigma_{ij}(x))_{1\le i,j \le n}$ is symmetric definite positive.
 For $f,g \in C^\infty(\mathbb{R}^n)$, we consider the following differential bilinear form
 \[
 \mathcal{T}(f,g) =\sum_{i,j=1}^n \sigma_{ij}(x) \frac{\partial f}{\partial x_i}\frac{\partial g}{\partial x_j},
 \]
 and, to simplify the notations, we will simply denote $\mathcal{T}(f,f):=\mathcal{T}(f)$.

\

Our basic assumption on the couple $(L,\mathcal{T})$ is the following:

\begin{assumption}\label{assumption1}
We  assume that there exists a smooth function $U$ and a constant $C>0$ such that
\begin{itemize}
 \item[i)] $U \ge 1$
 \item[ii)] $\mathcal{T}( U ) \le C U$
 \item[iii)] $LU \le C U$ 
 \item[iv)]  $\{ U \le m \}$ is a compact set for every $m$. 
\end{itemize}
\end{assumption}

This assumption (more precisely i),  iii) and iv) )   classically implies that $L$ is the generator of a stochastically complete Markov semigroup $(P_t)_{t \ge 0}$ that uniquely solves the parabolic Cauchy problem
\[
\begin{cases}
\frac{\partial \phi}{\partial t} =L\phi \\
\phi(0,x)=f(x)
\end{cases}
\]
in $L^\infty$. By hypoellipticity of $L$, this semigroup  maps the set of smooth and compactly supported functions $C_0^\infty(\M)$ into the set of smooth functions.

\

We can associate to the quadratic form $\mathcal{T}$ a distance $d$ defined as follows:
\[
d(x,y)=\sup \left\{ |f(x) -f(y)| , \mathcal{T}(f) \le 1, f \in C^\infty(\R^n) \right\}. 
\] 
This distance induces the usual topology on $\R^n$ and Assumption \ref{assumption1} (more precisely i),  ii) and iv) ) implies that $(\R^n,d)$ is a complete length metric space. 

\

Denote $\mathcal{P}(\mathbb{R}^n)$ the set of probability measures on $\R^n$.  For $1 \le p \le \infty$, the $L^p$-Wasserstein distance of two probability measures $\nu_1$ and $\nu_2$ on $\mathbb{R}^n$ is defined by
 $$
 W_p(\nu_1,\nu_2)  = \inf_{\pi \in \Pi} \| d \|_{L^p(\pi)}
$$
where the infimum is taken over  the set of probability measures $\Pi$ on $\mathbb{R}^n\times \mathbb{R}^n$ whose marginals are respectively $\nu_1$ and  $\nu_2$.

Finally, we introduce the following  bilinear form: For $f \in C^\infty(\mathbb{R}^n)$,
\[
\mathcal{T}_2 (f)=\frac{1}{2} ( L \mathcal{T}(f) - 2 \mathcal{T}( f , Lf)).
\]
Our main theorem is the following:
\begin{theorem}\label{contra1}
Let $K \in \mathbb{R}$. The following statements are equivalent:
\begin{enumerate}
\item For every $f \in C_0^\infty(\mathbb{R}^n)$,
\[
\mathcal{T}_2 (f) \ge -K \mathcal{T}(f).
\]
\item For every $f \in C_0^\infty(\mathbb{R}^n)$,
\[
\mathcal{T}( P_t f ) \le e^{2Kt} P_t(\mathcal{T}(f)). 
\]
\item  For every $\mu,\nu \in \mathcal{P}(\mathbb{R}^n)$, and $t \ge 0$,
\[
W_{2} (P^*_t \mu , P^*_t \nu) \le e^{Kt} W_{2} (\mu,\nu).
\]
\end{enumerate}
\end{theorem}

\begin{proof}
The equivalence between $2)$ and $3)$ is a consequence of Kuwada's duality \cite{Kuwada}.  

We now prove that 1) is equivalent to 2). Let us assume that 1) holds.
If $\mathcal{T}$ is compactly supported, the argument is standard and easy. Indeed, fix $T>0$ and consider the functional
 \[
\phi (t)= P_t(\mathcal{T}( P_{T-t} f )).
\]
Differentiating $\phi$ yields
\[
\phi' (t)= 2 P_t(\mathcal{T}_2(P_{T-t} f))\ge -2K \phi(t),
\]
and 2) easily follows from Gronwall's lemma.

When $\mathcal{T}$ is not compactly supported, since there is no a priori bound on $\mathcal{T}( P_{T-t} f )$, $P_t(\mathcal{T}(P_{T-t} f ))$ may not be well-defined and  we need to use a careful localization argument that relies on the existence of the function $U$. We adapt some arguments  by F.Y. Wang \cite{FYW2} to work out the localization.

Consider a smooth and decreasing function $h: \R_{\ge 0} \to \R$ such that $h=1$ on $[0,1]$ and $h=0$ on $[2,+\infty)$. Denote then $h_n=h \left( \frac{U}{n} \right)$ and consider the compactly supported diffusion operator
\[
L_n=h^2_n L.
\]
Since $L_n$ is compactly supported, a Markov semigroup $P_t^n$ with generator  $L_n$ is easily constructed as the unique bounded  solution of  $\frac{\partial P_t^n f}{\partial t} =L_nP^n_t f$, $f \in L^\infty$. Then, for every bounded $f$, pointwise
\[
P_t^n f \to P_t f , \quad n \to \infty.
\]
We fix $t >0$, $n \ge 1$  and $f \in C^\infty(\R^n)$ compactly supported inside the set $\{ U \le n \}$. Consider the functional defined for $s \in [0,t]$ and evaluated at a fixed point $x_0$ in the set $\{ U \le n \}$:
\[
\Phi_n(s) =P^n_s ( \mathcal{T}( P^n_{t-s} f )).
\]
We have
\[
\Phi'_n(s) =P^n_s (L_n  \mathcal{T}( P^n_{t-s} f )-2 \mathcal{T}( L_n P^n_{t-s} f,   P^n_{t-s} f) ).
\]
Now, observe that by assumption, and denoting $K^{-}$ the negative part of $K$, 
\begin{align*}
 & L_n  \mathcal{T}( P^n_{t-s} f )-2 \mathcal{T}( L_n P^n_{t-s} f,   P^n_{t-s} f)  \\
   =& 2h_n^2 \mathcal{T}_2 (  P^n_{t-s} f, P^n_{t-s} f) -4h_nL P^n_{t-s} f  \mathcal{T}( h_n  ,  P^n_{t-s} f) \\
 \ge & -2K h_n^2 \mathcal{T}( P^n_{t-s} f ) -4h_nL P^n_{t-s} f  \mathcal{T}( h_n  ,  P^n_{t-s} f)  \\
  \ge &  -2K h_n^2 \mathcal{T}( P^n_{t-s} f ) -4 P^n_{t-s} L_n f   \mathcal{T}( h_n  ,  P^n_{t-s} f) \\
  \ge & -2K h_n^2 \mathcal{T}( P^n_{t-s} f )-4 \| Lf \|_\infty \sqrt{ \mathcal{T}( \ln h_n) } \sqrt{ \mathcal{T} (P^n_{t-s} f)}| \\
   \ge &   -(2K^{-}+2)   \mathcal{T}( P^n_{t-s} f )-2 \| Lf \|^2_\infty \mathcal{T}( \ln h_n) . 
\end{align*}
The term $\mathcal{T}( \ln h_n)$ can be estimated as follows inside the set $\{ U \le 2n \}$
\[
\sqrt{ \mathcal{T}( \ln h_n) } =- \frac{1}{n h_n } h' \left( \frac{U}{n} \right) \sqrt{ \mathcal{T} (U)} \le \frac{C}{ h_n },
\]
where $C$ is a constant independent from $n$. On the other hand a direct computation and the assumptions on $U$ show that
\[
L_n \left( \frac{1}{h_n^2} \right) \le \frac{C}{h^2_n},
\]
where, again,  $C$ is a constant independent from $n$. This last estimate  implies
\[
P^n_s \left( \frac{1}{h_n^2} \right) \le  \frac{e^{Cs}}{h^2_n}.
\]
Putting the pieces together we end up with a differential inequality
\[
\Phi'_n(s) \ge   -(2K^{-}+2) \Phi_n(s)-C,
\]
where $C$ now depends on $f$ and  $t$, but still does not depend on $n$. Integrating this inequality from $0$ to $t$, yields a bound of the type
\[
\mathcal{T}( P^n_t f ) \le C^2,
\]
where  $C$ depends on $f$ and  $t$. This bounds holds uniformly on the set $\{ U \le n \}$. 
\

We now pick any $x,y \in \M, f\in C_0^\infty(\M)$ and $n$ big enough so that $x,y \in \{ U \le n \}$ and $\mathbf{Supp} (f) \subset  \{ U \le n \}$. We have from the previous inequality
\[
| P^n_t f (x) -P^n_t f (y) | \le C d(x,y),
\]
and thus, by taking the limit when $n \to \infty$,
\[
| P_t f (x) -P_t f (y) | \le C d(x,y).
\]
We therefore reach the important conclusion that $P_t$ transforms $C_0^\infty(\M)$ into a subset of the set of smooth and Lipschitz functions. With this conclusion in hands, we can now run the usual Bakry-\'Emery argument.

Let $f \in C^\infty_0(\mathbb{M})$,  and $T>0$, and consider the function
\[
\phi (x,t)=\mathcal{T}( P_{T-t} f )(x),
\]
We have
\[
L \phi+\frac{\partial \phi}{\partial t} =2 \mathcal{T}_2(  P_{T-t} f) \ge -2K \phi .
\]
Since we know that $\phi$ is bounded, we can use  the parabolic comparison principle in $L^\infty$  to conclude, thanks to Gronwall's inequality,
\[
\mathcal{T}( P_t f ) \le e^{2Kt} P_t(\mathcal{T}(f)).
\]
Finally, 2) implies 1) is easy to prove by differentiating the inequality 
\[
\mathcal{T}( P_t f ) \le e^{2Kt} P_t(\mathcal{T}(f))
\]
at $t=0$.
\end{proof}

As an immediate corollary, we deduce the following generalization of the Bakry-\'Emery criterion for ergodicity.

\begin{corollary}
Assume that there exists a constant $K>0$ such that for every $f \in C_0^\infty(\mathbb{R}^n)$,
\[
\mathcal{T}_2 (f) \ge K \mathcal{T}(f).
\]
Then, there exists a unique $\mu \in \mathcal{P}(\mathbb{R}^n)$  such that for every $t \ge 0$, $P^*_t \mu=\mu$. Moreover, for every $\nu \in \mathcal{P}(\mathbb{R}^n)$
\[
W_{2} (P^*_t \nu , \mu) \le e^{-Kt} W_{2} (\mu,\nu).
\]

\end{corollary}

Under further conditions, we can also prove that the invariant measure $\mu$ satisfies a Poincar\'e inequality. We recall that the \textit{carr\'e du champ operator} associated with $L$ is defined as:
\[
\Gamma(f)=\frac{1}{2} (Lf^2 -2fLf), \quad f \in C^\infty (\R^n).
\]

\begin{proposition}
Assume that there exists a constant $a >0$ such that for every $f \in C^\infty(\R^n)$,
\[
\Gamma(f) \le a \mathcal{T}(f)
\]
and that there exists a constant $K>0$ such that for every $f \in C^\infty(\mathbb{R}^n)$,
\[
\mathcal{T}_2 (f) \ge K \mathcal{T}(f).
\]
Then, the unique invariant probability measure $\mu$ of $L$ satisfies the Poincar\'e inequality
\[
\int f^2 d\mu -\left( \int f d\mu \right)^2 \le \frac{a}{K} \int \mathcal{T} (f) d\mu, \quad f \in C_0^\infty(\mathbb{R}^n).
\]
\end{proposition}

\begin{proof}
We can adapt the classical semigroup interpolation method by Bakry-\'Emery.
\begin{align*}
\int f^2 d\mu -\left( \int f d\mu \right)^2 &=-\int_0^{+\infty} \frac{d}{dt} \int (P_t f)^2 d\mu dt \\
 &=2\int_0^{+\infty} \int \Gamma(P_t f) d\mu dt \\
 &\le 2a\int_0^{+\infty}  \int \mathcal{T}(P_t f) d\mu dt \\
 &\le 2a\int_0^{+\infty}  \int e^{-2Kt} P_t \mathcal{T}(f) d\mu dt \\
 &\le 2a \int_0^{+\infty}   e^{-2Kt} dt \int \mathcal{T}(f) d\mu
\end{align*}
\end{proof}

\

We shall study in details examples of application of the above theorems in Sections 2.2, 2.3 and 2.4, however as an appetizer we give the following general class of examples. It shows that Theorem \ref{contra1} may be applied in very degenerate hypoelliptic situations.

Consider on $\mathbb{R}^n$ a hypoelliptic diffusion operator of the form
\[
L=\sum_{i=1}^n b_i(x) \frac{\partial }{\partial x_i} +\sum_{i,j=1}^n a_{ij} \frac{\partial^2 }{\partial x_i \partial x_j}  
\]
where the $b_i$'s are smooth functions and $A=(a_{ij})$ is a constant non negative symmetric matrix which does not need to be positive definite. If we assume $b=(b_1,\cdots,b_n)$ to be Lipschitz, then the function $U(x)=1+\| x \|^2 $ is  such that, for some constant $C>0$, $LU \le C U$ and $\| \nabla U \|^2 \le C U$. Therefore Assumption \ref{assumption1} is satisfied.

\

Let us denote by $(\partial b)_y$ the Jacobian matrix of $b$ at a point $y \in \R^n$ and by $P_t$ the semigroup generated by $L$. We have then the following theorem:

\begin{theorem}
 Assume that there exists a constant positive definite matrix $\Sigma$ and a constant $a>0$ such that for every $x,y \in \mathbb{R}^n$,
\[
-\langle \Sigma (\partial b)_y x , x \rangle \ge a \| x \|^2.
\]
Then, there exists a unique $\mu \in \mathcal{P}(\mathbb{R}^n)$  such that for every $t \ge 0$, $P^*_t \mu=\mu$. Moreover, there exist constants $C_1,C_2>0$  such that for every $\nu \in \mathcal{P}(\mathbb{R}^{n})$, and $t \ge 0$,
\[
W_{2} (P^*_t \nu , \mu) \le C_1 e^{-C_2t} W_{2} (\mu,\nu),
\]
where $W_2$ is the Wasserstein distance associated to the Euclidean distance on $\mathbb{R}^n$. Moreover, $\mu$ satisfies a Poincar\'e inequality
\[
\int f^2 d\mu -\left( \int f d\mu \right)^2 \le C_3 \int \| \nabla f \|^2 d\mu, \quad f \in C_0^\infty(\mathbb{R}^n),
\]
for some constant $C_3>0$.
\end{theorem}
\begin{proof}
Let $\sigma$ be a positive definite square root of $\Sigma$ and consider the gradient
\[
\mathcal{T}(f)=\| \sigma \nabla f \|^2.
\]
An easy computation shows that
\[
\mathcal{T}_2(f) \ge - \langle \Sigma \partial b \nabla f , \nabla f \rangle.
\]
The result follows then from Theorem \ref{contra1} because the distance associated to $\mathcal{T}$ is equivalent to the Euclidean distance.
\end{proof}

The main problem in practice will be to find the matrix $\Sigma$. In the case of the kinetic Fokker-Planck equation, the existence of the matrix $\Sigma$ is a non trivial problem that we address in the following section. 

\subsection{Example 1:  The kinetic Fokker-Planck equation}

In this section we study the kinetic Fokker-Planck equation which is an important example of equation to which the methods apply. We, in particular, prove for the first time contraction in the Wasserstein space for this equation.

Let $V:\mathbb{R}^n \to \mathbb{R}$ be a smooth function. The kinetic Fokker-Planck equation with confinement potential $V$ is the parabolic  partial differential equation:
\begin{equation}\label{FP2}
\frac{\partial h}{\partial t}=\Delta_v h - v \cdot \nabla_v h+\nabla_x V \cdot \nabla_v h -v\cdot \nabla_x h , \quad (x,v) \in \mathbb{R}^{2n}.
\end{equation}
This equation is Kolmogorov-Fokker-Planck equation associated to the stochastic differential equation
\[
\begin{cases}
dx_t =v_t dt \\
dv_t=-v_t dt -\nabla V (x_t) dt +dB_t,
\end{cases}
\]
where $(B_t)_{ t\ge 0}$ is a Brownian motion in $\mathbb{R}^n$. 

The operator
\[
L=\Delta_v  - v \cdot \nabla_v +\nabla_x V \cdot \nabla_v -v\cdot \nabla_x 
\]
admits for invariant measure 
\[
d\mu=e^{-V(x)-\frac{\| v \|^2}{2}} dxdv.
\]
It is readily checked that $L$ is not symmetric with respect to $\mu$. The operator $L$ is hypoelliptic and the generator of a strongly continuous sub-Markov semigroup $(P_t)_{t \ge 0}$. If we assume that the Hessian $\nabla^2 V$ is bounded, which we do in the sequel, then  the semigroup $P_t$ generated by $L$ is Markovian . 

Observe that since $\nabla V$ is Lipschitz, the function $U(x,v)=1+\| x \|^2 +\| v\|^2$ is  such that, for some constant $C>0$, $LU \le C U$ and $\| \nabla U \|^2 \le C U$. Therefore Assumption \ref{assumption1} is satisfied.

\

The problem of convergence to an equilibrium for solutions of the kinetic Fokker-Planck equation has attracted a lot of interest in the literature and many approaches have been proposed.

A functional analytic approach, based on previous ideas by Kohn and H\"ormander, relies on delicate spectral localization tools to study the resovent and prove exponential convergence to equilibrium with explicit bounds on the rate. For this approach, we refer to Eckmann and Hairer \cite{EH}, H\'erau and Nier \cite{HN1}, and  Heffer and Nier \cite{HN2}. 

\

L. Wu in \cite{Wu}, Mattingly, Stuart and Higham in \cite{Matt} and Bakry, Cattiaux and Guillin in \cite{BCG} use  Lyapunov functions method and probabilistic tools to prove  exponential convergence to equilibrium.

\

One of the most general results is due to  Villani (see Theorem \ref{Villani}). Villani in his memoir \cite{Villani1} introduces the concept of hypocoercivity and derives very general sufficient conditions ensuring the convergence to an equilibrium. The main strategy, already implicit in the work by Talay \cite{talay} is to work in a suitable Hilbert space associated to the equation and to find in this Hilbert space a nice norm which is equivalent to the original one, but with respect to which convergence to equilibrium is easy to obtain; We refer to Section 4.1 in \cite{Villani1} for a more precise description. The work by Villani has recently been revisited in \cite{baudoin-bakry} and we summarize the approach of \cite{baudoin-bakry} in Section 3 of the present paper.

\

In this section, by using the results of the previous section , we prove convergence to equilibrium  and moreover prove a contraction property in the Wasserstein space.  Our assumptions are admittedly quite strong  on the potential $V$, but the advantage of the method is not to use in any way the knowledge of the invariant measure.

\

We denote by $W_2$ the usual $L^2$ Wasserstein distance associated to the Euclidean metric on $\mathbb{R}^{2n}$.

\begin{theorem}\label{plk}
Assume that there exist constants $m,M>0$ such that $$m \le \nabla^2 V \le M$$ and $\sqrt{M} -\sqrt{m} \le 1$. Then, there exist constants $C_1,C_2>0$  such that for every $\mu,\nu \in \mathcal{P}(\mathbb{R}^{2n})$, and $t \ge 0$,
\[
W_{2} (P^*_t \mu , P^*_t \nu) \le C_1 e^{-C_2t} W_{2} (\mu,\nu).
\]
\end{theorem}

\begin{proof}
The idea is to work with a gradient on $\mathbb{R}^{2n}$ which does not come from the standard Euclidean structure and apply then Theorem \ref{contra1}. Let $\alpha,\beta,\gamma,\delta \in \mathbb{R}$ be constants to be chosen later and consider the gradient
\[
\mathcal{T}(f)=\sum_{i=1}^{n} \left( \alpha \frac{\partial f}{\partial x_i} +\beta  \frac{\partial f}{\partial v_i}\right)^2 +\left( \gamma \frac{\partial f}{\partial x_i} +\delta  \frac{\partial f}{\partial v_i}\right)^2.
\]
We denote as before
\[
\mathcal{T}_2 (f)=\frac{1}{2} ( L \mathcal{T}(f) - 2 \mathcal{T} (f , Lf)).
\]
If we can prove that we can chose $\alpha,\beta,\gamma,\delta \in \mathbb{R}$ such that for some constant $\rho >0$, 
\begin{align}\label{con2}
\mathcal{T}_2 (f) \ge \rho \mathcal{T}(f),
\end{align}
then the proof of the theorem is a consequence of Theorem \ref{contra1} since all norms are equivalent on $\mathbb{R}^{2n}$. 

\

We can write $\mathcal{T}$ in the form
\[
\mathcal{T}(f)=\sum_{i=1}^{2n} (Z_i f)^2
\]
with
\begin{align*}
Z_i=
\begin{cases}
\alpha \frac{\partial f}{\partial x_i} +\beta  \frac{\partial f}{\partial v_i} , 1 \le i \le n \\
 \gamma \frac{\partial f}{\partial x_{n-i}} +\delta  \frac{\partial f}{\partial v_{n-i}}, n+1 \le i \le 2n.
\end{cases}
\end{align*}

 and 
 \[
L=\sum_{i=1}^n X_i^2 +X_0+Y,
\]
where $X_i=\frac{\partial}{\partial v_i}$, $X_0=- v \cdot \nabla_v$ and $Y=\nabla V \cdot \nabla_v -v\cdot \nabla_x $.  We have then
\begin{align*}
\mathcal{T}_2 (f)&=\frac{1}{2} ( L \mathcal{T}(f) - 2 \mathcal{T} (f , Lf)) \\
 &=\frac{1}{2}\left(L\left(\sum_{i=1}^{2n} (Z_if)^2\right)-2\sum_{i=1}^{2n} Z_if Z_iLf\right)\\
 &=\sum_{i=1}^{2n}\sum_{j=1}^n (X_jZ_i f)^2 +\sum_{i=1}^{2n} Z_if [L,Z_i]f \\
 &=\sum_{i=1}^{2n} \sum_{j=1}^n (X_jZ_i f)^2+\sum_{i=1}^{2n} Z_if [X_0,Z_i]f +\sum_{i=1}^{2n} Z_if [Y,Z_i]f.
\end{align*}

As a consequence we obtain
\[
\mathcal{T}_2 (f) \ge \sum_{i=1}^{2n} Z_if [X_0,Z_i]f +\sum_{i=1}^{2n} Z_if [Y,Z_i]f.
\]

We now easily compute
\begin{align*}
[X_0,Z_i]=
\begin{cases}
\beta  \frac{\partial }{\partial v_i} , 1 \le i \le n \\
\delta  \frac{\partial }{\partial v_{i-n}}, n+1 \le i \le 2n.
\end{cases}
\end{align*}
and 
\begin{align*}
[Y,Z_i]=
\begin{cases}
\beta  \frac{\partial }{\partial x_i}-\alpha\sum_{j=1}^n  \frac{\partial^2 V}{\partial x_i \partial x_j} \frac{\partial }{\partial v_i}  , 1 \le i \le n \\
\delta  \frac{\partial }{\partial x_{i-n}}-\gamma\sum_{j=1}^n  \frac{\partial^2 V}{\partial x_i \partial x_j} \frac{\partial }{\partial v_{i-n}} , n+1 \le i \le 2n.
\end{cases}
\end{align*}
As a consequence, we obtain after straightforward computations
\begin{align*}
\mathcal{T}_2 (f) \ge  &  (\beta^2+\delta^2) \| \nabla_v f \|^2+(\alpha \beta +\gamma \delta) \| \nabla_x f \|^2+ (\beta^2 +\delta^2 +\alpha \beta +\gamma \delta)  \nabla_v f  \cdot  \nabla_x f \\
 & -(\alpha^2+\gamma^2) \nabla^2 V ( \nabla_v f, \nabla_x f)-(\alpha \beta +\gamma \delta) \nabla^2 V ( \nabla_v f, \nabla_v f).
\end{align*}
The right-hand side of the above inequality can be seen as a bilinear form on $\mathbb{R}^{2n}$ applied to $\nabla f= ( \nabla_x f, \nabla_v f)$. We want this form to be definite positive.

\

We first chose  $\alpha,\beta,\gamma,\delta$ such that $\alpha^2 +\gamma^2 =1$ and denote
\[
a=\alpha \beta +\gamma \delta, \quad  b =\beta^2+\delta^2.
\]
Observe that the only constraint on $a,b$ is that $a>0$ and  $ a ^2 \le b$.
A sufficient condition for the bilinear form to be definite positive is that for any eigenvalue $\lambda$ of $\nabla^2 V$, we have
\[
(a+b-\lambda)^2 < 4a (b-a \lambda).
\]
This inequality is equivalent to
\[
-\sqrt{\kappa^2-\theta^2} +\kappa < \lambda < \sqrt{\kappa^2-\theta^2} +\kappa,
\]
where 
\[
\kappa=a+b -2a^2, \quad \theta=b-a.
\]
We thus want to  chose $\alpha, \beta,\gamma,\delta$ in such a way that
\[
-\sqrt{\kappa^2-\theta^2} +\kappa=m, \sqrt{\kappa^2-\theta^2} +\kappa=M.
\]
This condition is equivalent to
\begin{align*}
\begin{cases}
\kappa=\frac{m+M}{2} \\
\theta^2=mM
\end{cases}
\end{align*}
We finally conclude by observing that the system
\begin{align*}
\begin{cases}
a+b -2a^2=\frac{m+M}{2} \\
b-a=\sqrt{mM}
\end{cases}
\end{align*}
has a solution $0<a^2 \le b$ as soon as $\sqrt{M}-\sqrt{m} \le 1$.
\end{proof}

\subsection{Example 2: Kolmogorov type operators on foliated manifolds}

To apply Theorem \ref{contra1} in concrete situations, the main problem is: 

\

\textit{
Given the operator $L$, how can we find a bilinear form $\mathcal{T}$ that satisfies 
\[
\mathcal{T}_2 (f) \ge K \mathcal{T}(f) 
\]
for some constant $K$?}

\

We show in this section that the geometry of  foliations can be useful to find a canonical $\mathcal{T}$ once a convenient foliation associated to $L$ is found.

\

Let $\M$ be a smooth, connected  manifold with dimension $n+m$. We assume that $\bM$ is equipped with a Riemannian foliation $\mathcal{F}$ with  $m$-dimensional leaves. We denote by $\Dv$ the vertical Laplacian of this foliation.

\begin{definition}
We call Kolmogorov type operator, a hypoelliptic diffusion operator $L$ on $\M$ that can be written as
\[
L=\Dv+Y,
\]
where $Y$ is a smooth vector field on $\M$.
\end{definition}

The simplest example of such an operator was studied by Kolmogorov himself. It is the operator
\[
L=\frac{\partial^2}{\partial v^2}+v \frac{\partial}{\partial x}.
\]
Then, by considering the trivial foliation on $\R^2$ whose leaves are the lines $\{(0,v), v \in \R \}$ , we can write $L=\Dv+Y$ with $\Dv=\frac{\partial^2}{\partial v^2}$ and $Y=v \frac{\partial}{\partial x}$. 

\

More interesting is the case of the kinetic of the Fokker-planck equation that was already considered before. Consider on $\R^{2n}=\{(v,x), v \in \R^n,  x \in \R^n \}$, the operator
\[
L=\Delta_v  - v \cdot \nabla_v +\nabla_x V \cdot \nabla_v -v\cdot \nabla_x,
\]
where $V: \R^n \to \R$ is a smooth potential. We can obviously write
\[
L=\Dv+Y,
\]
where $\Dv=\Delta_v$ and $Y=- v \cdot \nabla_v +\nabla_x V \cdot \nabla_v -v\cdot \nabla_x$, and consider the trivial foliation on $\R^{2n}$ where the leaves are given by the sets $\{(0,v), v \in \R^n \}$. 
\

Consider a Kolmogorov type operator 
\[
L=\Dv+Y,
\]
and assume that the Riemannian foliation $\mathcal{F}$ is totally geodesic with a bundle like metric (see \cite{B2} for a detailed description of this setting).

\

To use the results of Section 2, we  assume that there exists a function $W$ such that $W \ge 1$, $\| \nabla W \|^2 \le C W$, $LW \le C W$ for some constant $C>0$ and $\{ W \le m \}$ is compact for every $m$.  

\

We denote by $\mathbf{Ric}_\mathcal{V}$ the  Ricci curvature of the leaves and we denote by $DY$ the tensor defined by $DY(U,V)=\langle D_U Y, V \rangle$ where $D$ is the Levi-Civita connection of the Riemannian metric. 

\begin{theorem}\label{GB18}
Let us assume that for some $K > 0$,
\[
\mathbf{Ric}_\mathcal{V}  -DY  \ge K.
\]
Then, there exists a unique probability measure $\mu \in \mathcal{P}(\mathbb{M})$  such that for every $t \ge 0$, $P^*_t \mu=\mu$. Moreover, for every $\nu \in \mathcal{P}(\mathbb{R}^n)$
\[
W_{2} (P^*_t \nu , \mu) \le e^{-Kt} W_{2} (\mu,\nu),
\]
where $W_{2}$ is the Wasserstein distance associated to the Riemannian distance on $\M$.

\end{theorem}

\begin{proof}
 If $f \in C^\infty(\M)$, we denote
\[
\mathcal{T}_2(f)=\frac{1}{2} \left( L (\| \nabla f \|^2) -2\langle \nabla f , \nabla L f \rangle \right),
\]
where $\nabla$ is the Riemannian gradient.  It is proved in \cite{B2} that for every $f \in C^\infty(\M)$,
\[
\mathcal{T}_2(f) \ge (\mathbf{Ric}_\mathcal{V}  -DY ) (\nabla f, \nabla f).
\]
As a consequence,
\[
\mathbf{Ric}_\mathcal{V}  -DY  \ge K
\]
implies that
\[
\mathcal{T}_2(f) \ge K \| \nabla f \|^2.
\]
\end{proof}

\subsection{Example 3: The horizontal Laplacian on totally geodesic Riemannian foliations}

In this section, we study a class of hypoelliptic diffusion operators which are relevant in sub-Riemannian geometry. Problems of convergence for the semigroup associated with those operators have already been studied in \cite{BB}. We revisit those results in view of our new approach.

\ 

We consider a smooth connected $n+m$ dimensional manifold $\M$ which is equipped with a Riemannian foliation with a bundle like  metric $g$ and totally $m$ dimensional geodesic leaves. We moreover assume that the metric $g$ is complete and that the horizontal distribution $\mathcal{H}$ of the foliation is bracket-generating and Yang-Mills (see \cite{B2} for a detailed presentation of this framework). The hypothesis that $\mathcal{H}$ is bracket generating implies that the horizontal Laplacian $\Delta_{\mathcal{H}}$ is subelliptic and the completeness assumption on $g$ implies that $\Delta_{\mathcal{H}}$ is essentially self-adjoint on the space of smooth and compactly supported functions. The heat semigroup generated by $\Delta_{\mathcal{H}}$ will be denoted by $P_t$. We denote by $\mu$ the Riemannian reference measure on $\M$.

\

As before, the sub-bundle $\mathcal{V}$ formed by vectors tangent to the leaves is referred  to as the set of \emph{vertical directions}. The sub-bundle $\mathcal{H}$ which is normal to $\mathcal{V}$ is referred to as the set of \emph{horizontal directions}.   The Riemannian gradient will denoted $\nabla$ and the horizontal gradient, which is the projection of $\nabla$ onto $\mathcal{H}$, by $\nabla_\mathcal{H}$. Likewise, $\nabla_\mathcal{V}$ will denote the vertical gradient.

\

The metric $g$ can be split as
\[
g=g_\mathcal{H} \oplus g_{\mathcal{V}},
\]
and  we introduce the one-parameter family of rescaled Riemannian metrics:
\[
g_{\varepsilon}=g_\mathcal{H} \oplus  \frac{1}{\varepsilon }g_{\mathcal{V}}, \quad \varepsilon >0.
\]
It is called the canonical variation of $g$. The Riemannian distance associated with $g_{\varepsilon}$ will be denoted by $d_\varepsilon$. It should be noted that $d_\varepsilon$, $\varepsilon >0$, form an increasing (as $\varepsilon\downarrow 0$) family of distances converging to the sub-Riemannian distance.

\

The following definition was introduced in \cite{BG}.

\begin{definition}
Let $K \in \mathbb{R}$, $\kappa \ge 0$, $\rho_2 \ge 0$. We say that the sub-Laplacian $\Delta_\mathcal{H}$ satisfies the generalized curvature dimension inequality $CD(K,\kappa, \rho_2, \infty)$ if for every $f \in C_0^\infty(\M)$ and every $\varepsilon >0$,
\[
\frac{1}{2} \Delta_{\mathcal{H}} \| \nabla f \|^2_\varepsilon -\langle \nabla f , \nabla \Delta_\mathcal{H} f \rangle_\varepsilon \ge  \left( K-\frac{\kappa}{\varepsilon}\right) \| \nabla f \|^2_\mathcal{H}+\rho_2 \| \nabla f \|^2_\V.
\]
\end{definition}

Geometric conditions ensuring that the sub-Riemannian curvature inequality is satisfied were studied in  \cite{BG,BKW}. The parameter $K$ is a lower bound on a sub-Riemannian Ricci tensor and the parameters $\kappa,\rho_2$ are respectively upper and lower bounds on torsion related tensors. Observe that the parameter $\rho_2$ is always $\ge 0$. Since
\[
 \| \nabla f \|^2_\varepsilon =\| \nabla f \|^2_\mathcal{H}+\varepsilon \| \nabla f \|^2_\V,
\]
we see that the generalized curvature dimension inequality implies a criterion similar to the one of the previous section, and thus Theorem 1.1 applies. In particular, we obtain the following result:

\begin{proposition}
Assume that the sub-Laplacian $\Delta_\mathcal{H}$ satisfies the generalized curvature dimension inequality $CD(K,\kappa, \rho_2, \infty)$ with $K,\rho_2 >0$, then the semigroup $P_t$ generated by  $\Delta_{\mathcal{H}}$ converges to equilibrium and moreover  for every $\varepsilon > \frac{\kappa}{K}$,  $\mu,\nu \in \mathcal{P}(\M)$, and $t \ge 0$,
\[
W^\varepsilon_{2} (P^*_t \mu , P^*_t \nu) \le e^{-\lambda_\varepsilon t} W^\varepsilon_{2} (\mu,\nu),
\]
where $W^\varepsilon_{2}$ is the $L^2$ Wasserstein distance associated to the distance $d_\varepsilon$ and $\nu_\varepsilon=\min \left\{ K-\frac{\kappa}{\varepsilon}, \frac{\rho_2}{\varepsilon} \right\} $
\end{proposition}

\section{Convergence in $H^1(\mu)$}

In this section, we summarize and revisit the approach in \cite{baudoin-bakry} to prove convergence to equilibrium. The method typically yields convergence under much weaker assumptions, however requires some informations about the invariant measure of the operator $L$ which may difficult to check in practice. The strength of the methods developed in Section 2, is that the invariant measure is not even assumed to exist.

\

As in Section 2, we consider on $\mathbb{R}^n$ a hypoelliptic diffusion operator 
\[
L=\sum_{i,j=1}^n a_{ij} (x) \frac{\partial^2}{\partial x_i \partial x_j } +\sum_{i=1}^n b_i (x) \frac{\partial}{\partial x_i},
\]
 with smooth coefficients $a_{ij},b_i$. Let us assume that $L$ admits an invariant probability measure $\mu$. We also consider a first order symmetric and positive definite bilinear form
 \[
 \mathcal{T}(f,g) =\sum_{i,j=1}^n \sigma_{ij}(x) \frac{\partial f}{\partial x_i}\frac{\partial g}{\partial x_j},
 \]

with smooth coefficients and define

\[
\mathcal{T}_2 (f)=\frac{1}{2} ( L \mathcal{T}(f) - 2 \mathcal{T}( f , Lf)).
\]

The basic assumptions are:

\begin{assumption}

\

\begin{itemize}
\item There exist constants $K_1,K_2 >0 $ such that
\[
\int \mathcal{T}_2 (f) d\mu \ge K_1 \int \mathcal{T}(f) d\mu  -K_2 \int \Gamma(f) d\mu,
\]
where $\Gamma$ is the  \textit{carr\'e du champ} operator associated to $L$.
\item The invariant probability measure $\mu$ satisfies a Poincar\'e inequality
\[
\int f^2 d\mu -\left( \int f d\mu \right)^2 \le C \int \mathcal{T} (f) d\mu.
\]
\end{itemize}
\end{assumption} 

Under these assumptions, provided that $L$ generates a nice semigroup $P_t$ and that the following computations can be justified, we can prove convergence to equilibrium as follows. Let $f$ such that $\int f d\mu =0$ and consider the functional
\[
\Phi(t)= \int \mathcal{T}(P_t f) d\mu+b \int (P_t f)^2 d\mu
\]
where $b$ is a positive number to be chosen later. Differentiating $\Phi$ yields
\begin{align*}
\Phi'(t)&=-2\int \mathcal{T}_2(P_t f) d\mu-2b \int \Gamma(P_t f) d\mu \\
 &\le -2K_1 \int \mathcal{T}(P_t f) d\mu+(2K_2-2b) \int \Gamma(P_t f) d\mu \\
 &\le -2K_1 \int \mathcal{T}(P_t f) d\mu+\frac{1}{C} (2K_2-2b)\int (P_t f)^2 d\mu \\
 \end{align*}
Therefore, by chosing a suitable $b$ we obtain for some $0< C'\le \min \{ 2K_1, 2/C \}$ the Gronwall's inequality
\[
\Phi(t) \le e^{-C' t} \Phi(0),
\]
which implies
\[
 \int \mathcal{T}(P_t f) d\mu+b \int (P_t f)^2 d\mu\le  e^{-C' t} \left( \int \mathcal{T}( f) d\mu+b \int f^2 d\mu\right),
\]
and thus convergence to equilibrium at an exponential rate.

\

It may not be easy to justify rigorously the above argument. However, in the    case of the kinetic Fokker-Planck equation, where

\[
L=\Delta_v  - v \cdot \nabla_v +\nabla_x V \cdot \nabla_v -v\cdot \nabla_x ,
\]

everything can be done rigorously (see \cite{baudoin-bakry}) and we exactly get the following result, originally due to Villani.

\begin{theorem}[Villani \cite{Villani1}, Theorem 35]\label{Villani}
Define $$H^1(\mu)=\{ f \in L^2(\mu), \| \nabla f \| \in L^2(\mu)\}.$$ Assume that there is a constant $c>0$ such that $\| \nabla^2 V \| \le c( 1+ \| \nabla V \|)$ and that the normalized invariant measure $d\mu=\frac{1}{Z}e^{-V(x)-\frac{\| v \|^2}{2}} dxdv$  is a probability measure that satisfies the classical Poincar\'e inequality
\[
\int_{\mathbb{R}^{2n}} \| \nabla f \|^2 d\mu \ge \kappa \left[ \int_{\mathbb{R}^{2n}} f^2 d\mu -\left( \int_{\mathbb{R}^{2n}} f d\mu\right)^2 \right].
\]
Then, there exist constants $C>0$ and $\lambda >0$ such that for every $f \in H^1(\mu)$, with $\int_{\mathbb{R}^{2n}} f d\mu=0$,
\[
\int_{\mathbb{R}^{2n}} (P_t f)^2 d\mu + \int_{\mathbb{R}^{2n}} \| \nabla P_t f \|^2 d\mu \le C e^{-\lambda t} \left( \int_{\mathbb{R}^{2n}} f^2 d\mu + \int_{\mathbb{R}^{2n}} \| \nabla  f \|^2 d\mu\right)
\]
\end{theorem}

\end{document}